\documentclass[11pt,leqno]{amsart}
\usepackage{amssymb}
\usepackage{mathrsfs}
\usepackage[all]{xy}
\usepackage{color}
\usepackage{soul}
\usepackage{hyperref}

\newcommand\N{{\mathbb N}}
\newcommand\R{{\mathbb R}}

\newcommand\C{{\mathbb C}}

\def\AA{{\mathcal A}}
\def\BB{{\mathcal B}}

\def\EE{{\mathcal E}}
\def\FF{{\mathcal F}}

\def\LL{{\mathcal L}}

\def\RR{{\mathcal R}}
\def\SS{{\mathcal S}}

\def\BBB{{\mathscr B}}
\def\CCC{{\mathscr C}}

\newtheorem{theo}{Theorem}

\newtheorem{lem}[theo]{Lemma}
\newtheorem{cor}[theo]{Corollary}
\newtheorem{rem}[theo]{Remark}

\newcommand{\beqn}{\begin{equation}}
\newcommand{\eeqn}{\end{equation}}
\newcommand{\bear}{\begin{eqnarray}}
\newcommand{\eear}{\end{eqnarray}}
\newcommand{\bean}{\begin{eqnarray*}}
\newcommand{\eean}{\end{eqnarray*}}

\title{Fractional Fokker-Planck equation}

\begin{document}

\author{{\sc Isabelle Tristani}}
\address{CEREMADE, Universit\'e Paris IX-Dauphine,
Place du Mar\'echal de Lattre de Tassigny, 75775 Paris Cedex 16, France. 
E-mail: {\tt tristani@ceremade.dauphine.fr }}

\date\today

\maketitle

\begin{abstract}
This paper deals with the long time behavior of solutions to a ``fractional Fokker-Planck" equation of the form $\partial_t f = I[f] + \text{div}(xf)$ where the operator $I$ stands for a fractional Laplacian. We prove an exponential in time convergence towards equilibrium in new spaces. Indeed, such a result was already obtained in a $L^2$ space with a weight prescribed by the equilibrium in~\cite{GI}. We improve this result obtaining the convergence in a $L^1$ space with a polynomial weight. To do that, we take advantage of the recent paper~\cite{GMM} in which an abstract theory of enlargement of the functional space of the semigroup decay is developed. 
\end{abstract}

\vspace{0.5cm}

\textbf{Mathematics Subject Classication (2010):} 47G20 Integro-differential operators; 35B40 Asymptotic behavior of solutions; 35Q84 Fokker-Planck equations.

\vspace{0.3cm}

\vspace{0.3cm}
\tableofcontents
\newpage

\section{Introduction} 
\label{sec:intro}
\setcounter{equation}{0}
\setcounter{theo}{0}


\subsection{Model and main result}

For $\alpha \in (0,2)$, we consider the following generalization of the  Fokker-Planck equation:

\beqn \label{eq:FP}
\partial_t f = -\left( - \Delta \right)^{\alpha/2} f + \text{div}(xf) , \quad \text{in} \,\, \R^d
\eeqn
with an initial data $f_0$. In the sequel, we will use the shorthand notations 
$$
I[f] =  -\left( - \Delta \right)^{\alpha/2} f \quad \text{and} \quad \LL f = I[f] + \text{div} (xf).
$$ 
The operator $(-\Delta)^{\alpha/2}$ is a fractional Laplacian, we first define it on the space of Schwartz functions $\SS(\R^d)$ and we then extend the definition to others functions. We refer to Section~\ref{sec:LF} for the exact definition and for properties. 

We also define here weighted $L^p$ spaces in the following way: for some given Borel weight function $m\ge0$ on $\R^d$, let us define $L^p(m)$, $1 \le p \le +\infty$, as the Lebesgue space associated to the norm 
$$
\| h \|_{L^p(m)} = \| h \, m \|_{L^p}.
$$

We now state our main result concerning the equation~(\ref{eq:FP}). 
\begin{theo} \label{th:main}
Let us consider $k \in (0,\alpha)$. For any $a \in (- \min(\lambda,k),0)$ (where $\lambda>0$ will be defined in Corollary~\ref{cor:entropie})  and for any initial data $f_0 \in L^1(\langle x \rangle^k)$, the solution~$f(t)$ of the equation~(\ref{eq:FP}) satisfies the following decay:
$$
\| f(t) - \mu \langle f_0 \rangle \|_{L^1(\langle x \rangle^k)} \leq C_a e^{at} \|f_0 - \mu \langle f_0 \rangle \|_{L^1(\langle x \rangle^k)}
$$
where $\langle f_0 \rangle = \int_{\R^d} f_0$  and for some constant $C_a>0$. 
\end{theo}

\medskip

\subsection{Known results}
The main references to mention here are the papers~\cite{BK} and~\cite{GI}. In these two papers, ``L\'{e}vy-Fokker-Planck equations" (the fractional Laplacian is replaced by a L\'{e}vy operator) are studied using the entropy production method. There is a proof of existence and uniqueness of a nonnegative steady state of mass $1$ of the associated stationary equation. Then, in a weighted $L^2$ space with a weight prescribed by the equilibrium, a convergence (with an exponential rate) of the solution of the full equation towards equilibrium is obtained. Let us give more details about these results. We first introduce the main tools used. 

Consider a smooth convex function $\Phi : \R^+ \rightarrow \R$ and $\mu$ positive such that $\int_{\R^d} \mu(x) \, dx =1$ and define the $\Phi$-entropy: for any nonnegative function $f$,
$$
\text{Ent}^\Phi_\mu (f) := \int_{\R^d} \Phi(f) \, \mu \, dx - \Phi\left( \int_{\R^d} f \, \mu \, dx \right).
$$
Jensen's inequality gives that $\text{Ent}^\Phi_\mu (f) \geq 0$. Let $f_0$ be an initial condition of a L\'{e}vy-Fokker-Planck equation or
of the classical Fokker-Planck equation:
\beqn \label{eq:FP0}
\partial_t f = \Delta f + \text{div}(xf) , \quad \text{in} \,\, \R^d
\eeqn
Then, let us introduce the quantity  
$
E_\Phi(f_0)(t) := \text{Ent}^\Phi_\mu \left( \frac{f(t)}{\mu} \right)
$
which is well-defined for any $t>0$. 

In the case of the classical Fokker-Planck equation~(\ref{eq:FP0}), by using functional inequalities as Poincar\'{e}, logarithmic Sobolev or $\Phi$-entropy inequalities, one obtains exponential decays to zero of $E_\Phi(f_0)$. Then, the solution $f$ of~(\ref{eq:FP0}) converges towards the steady state of mass $1$ in the sense of $\Phi$-entropy. Methods to prove such results are usually based on entropy/entropy-production tools. See~\cite{Bak, MR1845806, MR1842428, Chaf} for different methods and applications. 

In~\cite{BK}, Biler and Karch study L\'{e}vy-Fokker-Planck equations where the L\'{e}vy operators are Fourier multipliers associated to symbols $a(\xi)$ satisfying for some real number $\beta \in (0,2]$ 
$$
0 < \liminf_{\xi \rightarrow 0} \frac{a(\xi)}{|\xi|^\beta} \leq \limsup_{\xi \rightarrow 0} \frac{a(\xi)}{|\xi|^\beta} < \infty \quad \text{and} \quad 0< \inf \frac{a(\xi)}{|\xi|^2}.
$$
They prove that there exist $C>0$ and $\epsilon>0$ such that 
$$
E_{| \cdot|^2/2}(f_0)(t) \le C e^{-\epsilon t},
$$
which means that the solution converges towards equilibrium at an exponential rate in $L^2(\mu^{-1/2})$ where we denote $\mu$ the only steady state of mass $1$. They deduce a similar result in $L^2$ and finally, under some more restrictive regularity and decay assumptions on $f_0$, they prove that the exponential convergence holds in $L^1$. 

In~\cite{GI}, taking advantage of the paper~\cite{BK}, Gentil and Imbert prove an exponential decay of the $\Phi$-entropies for a class of convex functions $\Phi$ and for a larger class of operators which includes the fractional Laplacian.

In the present paper, we only consider the equation~(\ref{eq:FP}) but we are able to enlarge the space where we have a decay towards equilibrium with minimal assumptions on $f_0$. If we compare our result to the one obtained in~\cite{BK} for others operators defined above, we have to underline the fact that the result of convergence of the solution towards equilibrium in $L^1$ from~\cite{BK} requires additional assumptions on $f_0$ ($f_0$ must have finite moments of a large order), it is not the case in our main result where $f_0$ is only supposed to belong to $L^1(\langle x \rangle^k)$ with $k<\alpha$. 
\medskip

\subsection{Method of proof and outline of the paper}
The main outcome of the present paper is a result of decay towards equilibrium with an exponential rate of convergence in $L^1(\langle x \rangle^k)$ (with $k<\alpha$) for solutions of our equation~(\ref{eq:FP}). To do that, we adopt the same strategy as the one adopted in~\cite{GMM} by Gualdani, Mischler and Mouhot for the classical Fokker-Planck equation. 
Let us explain in more details this strategy. It is based on the theory of enlargement of the functional space of the semigroup decay developed in~\cite{GMM}. It enables to get a spectral gap in a larger space when we already have one in a smaller space. It applies to operators $\LL$ which can be splitted into two parts, $\LL=\AA+\BB$ with $\AA$ bounded and $\BB$ dissipative. Moreover, if we denote $e^{\BB t}$ the semigroup associated to the operator $\BB$, the semigroup $\left(\AA e^{\BB t}\right)$ is required to have some regularization properties. 
The fact that we can use this theory for our operator is based on two facts:
\begin{itemize}
\item we know from~\cite{GI} that our operator has a spectral gap in $L^2(\mu^{-1/2})$ where $\mu$ is the only steady state of mass $1$ of~(\ref{eq:FP}),
\item we are able to get a splitting satisfying the previous properties using computations based on properties of the fractional Laplacian.
\end{itemize}


In section~\ref{sec:LF}, we recall some technical tools about the fractional Laplacian that are useful in order to get a splitting of the operator. In section~\ref{sec:L2}, we state results from~\cite{GI} which are necessary to apply the abstract theorem of enlargement of spectral gap, which is reminded in Section~\ref{sec:abstractthm}. Finally, in Section~\ref{sec:L1}, we apply this theorem to obtain our main result on the convergence towards equilibrium of the solution of~(\ref{eq:FP}) in $L^1(\langle x \rangle^k)$ with $k<\alpha$. 

\medskip
\noindent \textbf{Acknowledgements} We would like to thank St\'{e}phane Mischler and Robert Strain for enlightened discussions and their help.
\bigskip


\section{Preliminaries on the fractional Laplacian} 
\label{sec:LF}
\setcounter{equation}{0}
\setcounter{theo}{0}


In this section, we recall some elementary properties of the fractional Laplacian that we will need through this paper. The usual reference for this kind of operators is Landkof's book~\cite{Land}. 

\subsection{Definition on $\SS(\R^d)$}

Let us consider $\alpha \in (0,2)$. The fractional Laplacian $\left( - \Delta \right)^{\alpha/2}$ is an operator defined on $\SS(\R^d)$  by: 
\beqn \label{eq:LF} 
\forall \, f \in \SS(\R^d), \quad 
\left(- \Delta \right)^{\alpha /2} f(x) =  \int_{\mathbb{R}^d} {\frac{f(x)-f(y)}{\left|x-y\right|^{d+\alpha}} dy}  . 
\eeqn
This definition has to be understood in the sense of principal value:
$$ 
\left(- \Delta \right)^{\alpha /2} f(x) = \lim_{\epsilon \rightarrow 0} \int_{\left|x-y\right| \geq \epsilon} {\frac{f(x)-f(y)}{\left|x-y\right|^{d+\alpha}} dy}  . 
$$

Due to the singularity of the kernel, the right hand-side of~(\ref{eq:LF}) is not well defined in general. However, when $\alpha \in (0,1)$, the integral is not really singular near $x$. Indeed, since $f \in \SS(\R^d)$, both $f$ and $\nabla f$ are bounded. We hence deduce the following inequality:
\begin{align*} 
\left| \int_{\mathbb{R}^d} {\frac{f(x)-f(y)}{\left|x-y\right|^{d+\alpha}} dy} \right| 
&\leq \| \nabla f \|_{L^\infty} \int_{\overline{B}(x, 1)} {\frac{dy}{\left|x-y\right|^{d+\alpha-1}}} + \| f \|_{L^\infty} \int_{\mathbb{R}^d \setminus \overline{B}(x, 1)} {\frac{dy}{\left|x-y\right|^{d+\alpha}}}.
\end{align*} 

When $\alpha \in (0,2)$, we can also write the fractional Laplacian with a non principal value integral. For any $f \in \SS(\R^d)$, we have 
\beqn \label{eq:LF2}
\forall \, x \in \mathbb{R}^d, \quad \left(- \Delta \right)^{\alpha /2} f(x) = -\frac{1}{2} \int_{\mathbb{R}^d} {\frac{f(x+y)+f(x-y) - 2 f(x)}{\left|y\right|^{d+\alpha}} \, dy}
\eeqn 
and this integral is well defined.

We can extend the integral definition of the fractional Laplacian to the following set of functions:
$$
 \left\{ f: \R^d \rightarrow \R, \, \, \int_{\R^d} \frac{|f(x)|}{1+|x|^{d+\alpha}} \, dx < \infty \right\}
$$
In particular, we can define $(-\Delta)^{\alpha/2} \langle x \rangle^k$ when $k<\alpha$.

\medskip
\subsection{Fractional Laplacian and Fourier transform}
Let us remind a well-known fact about the Fourier transform of the fractional Laplacian of a Schwartz function.
\begin{lem}
\label{lem:LFFourier}
There exists $C >0$ such that for any $f \in \mathcal{S}\left(\mathbb{R}^d\right)$, we have:
$$
\mathcal{F} \left(\left(- \Delta \right)^{\alpha /2} f\right)(\xi) = C \left|\xi\right|^\alpha \widehat{f}(\xi) . 
$$ 
\end{lem}

If $f$ is a Schwartz function, there is a singularity at $0$ in the Fourier transform of $(-\Delta)^{\alpha/2} f$. It implies a lack of decay at infinity for $(-\Delta)^{\alpha/2} f$ itself, $(-\Delta)^{\alpha/2} f$ is not a Schwartz function. We can prove that $(-\Delta)^{\alpha/2} f$ decays at infinity as $|x|^{-d-\alpha}$. 

We now mention a very useful property of the fractional Laplacian which can be seen as a sort of integration by parts. 
\begin{lem} \label{lem:IPP}
Let us consider $f$ and $g$ two Schwartz functions. Then, we have
$$
\int_{\R^d} (-\Delta)^{\alpha/2} f(x) \, g(x) \, dx = \int_{\R^d} f(x) \, (-\Delta)^{\alpha/2} g(x) \, dx.
$$
If $k<\alpha$, we can also prove that 
$$
\int_{\R^d} (-\Delta)^{\alpha/2} f (x) \, \langle x \rangle^k \, dx = \int_{\R^d} f(x) \, (-\Delta)^{\alpha/2} \langle x \rangle^k\, dx.
$$
\end{lem}

\medskip
\subsection{Fractional Laplacian and fractional Sobolev spaces}

Most of the time, fractional Sobolev spaces $H^s\left(\mathbb{R}^d\right)$ are defined in the following way: $H^s\left(\mathbb{R}^d\right)=W^{s,2}\left(\mathbb{R}^d\right)$ for $s \geq0$ is the set of functions $f \in L^2\left(\mathbb{R}^d\right)$ such that $\left[\left(1+\left|\cdot\right|^2\right)^{s/2} \, \widehat{f}\right]$ is also in $L^2\left(\mathbb{R}^d\right)$.
We remind here an equivalent definition which is going to be useful in what follows.

\begin{lem}
\label{lem:sobolevfrac}
Let us consider $s \in (0,1)$. We have:
$$
H^s\left(\mathbb{R}^d\right)=\left\{f \in L^2\left(\mathbb{R}^d\right) \, : \, \frac{\left|f(x) - f(y)\right|}{\left|x-y\right|^{\frac{d}{2}+s}} \in L^2\left(\mathbb{R}^d \times \mathbb{R}^d \right)\right\}. 
$$
We also have the following fact:
$$
\left\| \left(-\Delta\right)^{\alpha /2} f \right\|^2_{L^2\left(\mathbb{R}^d\right)} =  C \int_{\mathbb{R}^d} {\int_{\mathbb{R}^d} {\frac{\left|f(x) - f(y)\right|^2}{\left|x-y\right|^{d+\alpha}} \, dy} \, dx}
$$
for some $C>0$.
\end{lem}

\bigskip


\section{Theorem of enlargement of the functional space of the semigroup decay} 
\label{sec:abstractthm}
\setcounter{equation}{0}
\setcounter{theo}{0}


\subsection{Notations}

For a given real number $a \in \R$, we define the half complex plane
$$
\Delta_a := \left\{ z \in \C, \, \Re e \, z > a \right\}.
$$

\smallskip
For some given Banach spaces $(E,\|\cdot \|_E)$ and $(\EE,\| \cdot
\|_\EE)$ we denote by $\mathscr{B}(E, \EE)$ the space of bounded linear
operators from $E$ to $\EE$ and we denote by $\| \cdot
\|_{\mathscr{B}(E,\EE)}$ or $\| \cdot \|_{E \to \EE}$ the associated norm
operator. We write $\mathscr{B}(E) = \mathscr{B}(E,E)$ when $E=\EE$.
We denote by $\mathscr{C}(E,\EE)$ the space of closed unbounded linear
operators from $E$ to $\EE$ with dense domain, and $\mathscr{C}(E)=
\mathscr{C}(E,E)$ in the case $E=\EE$.

\smallskip
For a Banach space $X$ and   $\Lambda \in \mathscr{C}(X)$ we denote by $e^{\Lambda t}$, $t \ge
0$, its semigroup, by $\mbox{D}(\Lambda)$ its domain, by
$\mbox{N}(\Lambda)$ its null space and by $\mbox{R}(\Lambda)$ its range. We also
denote by $\Sigma(\Lambda)$ its spectrum, so that for any $z$ belonging to the resolvent set $\rho(\Lambda) :=  \C
\backslash \Sigma(\Lambda)$  the operator $\Lambda - z$ is invertible
and the resolvent operator
$$
\RR_\Lambda(z) := (\Lambda -z)^{-1}
$$
is well-defined, belongs to $\mathscr{B}(X)$ and has range equal to
$\mbox{D}(\Lambda)$.
We recall that $\xi \in \Sigma(\Lambda)$ is said to be an eigenvalue
if $\mbox{N}(\Lambda - \xi) \neq \{ 0 \}$. Moreover, an eigenvalue $\xi \in
\Sigma(\Lambda)$ is said to be isolated if
\[
\Sigma(\Lambda) \cap \left\{ z \in \C, \,\, |z - \xi| \le r \right\} =
\{ \xi \} \ \mbox{ for some } r >0.
\]
In the case when $\xi$ is an isolated eigenvalue, we may define
$\Pi_{\Lambda,\xi} \in \mathscr{B}(X)$ the associated spectral projector by
$$
\Pi_{\Lambda,\eta} := - {1 \over
  2i\pi} \int_{ |z - \xi| = r' } (\Lambda-z)^{-1} \, dz
$$
with $0<r'<r$. Note that this definition is independent of the value
of $r'$ as the application
$
\C \setminus \Sigma(\Lambda) \to \mathscr{B}(X)$, $z \to \RR_{\Lambda}(z)$ is holomorphic.
For any $\xi \in \Sigma(\Lambda)$ isolated, it is well-known (see~\cite{Kato} paragraph III-6.19) 
that $\Pi_{\Lambda,\xi}^2=\Pi_{\Lambda,\xi}$,  so that $\Pi_{\Lambda,\xi}$ is indeed a projector.

\smallskip
When moreover the so-called ``algebraic eigenspace" $\mbox{R}(\Pi_{\Lambda,\xi})$ is finite dimensional we say that
$\xi$ is a discrete eigenvalue, written as $\xi \in \Sigma_d(\Lambda)$. In that case,
$\RR_\Lambda$ is a meromorphic function on a neighborhood of $\xi$,
with non-removable finite-order pole $\xi$. 

\smallskip
Finally for any $a \in \R$ such that
\[
\Sigma(\Lambda) \cap \Delta_{a} = \left\{ \xi_1, \dots, \xi_k\right\}
\]
where $\xi_1, \dots, \xi_k$ are distinct discrete eigenvalues, we define
without any risk of ambiguity
\[
\Pi_{\Lambda,a} := \Pi_{\Lambda,\xi_1} + \dots \Pi_{\Lambda,\xi_k}.
\]

We shall also need the following definition on the convolution of semigroups. Consider some Banach spaces $X_1$, $X_2$ and $X_3$. For two given functions 
$$
S_1 \in L^1\left(\mathbb{R}^+ ; \BBB \left(X_1,X_2\right)\right) \quad \text{and} \quad S_2 \in L^1\left(\mathbb{R}^+ ; \BBB\left(X_2,X_3\right)\right), 
$$
the convolution $S_2 \ast S_1 \in L^1\left(\mathbb{R}^+ ; \BBB\left(X_1,X_3\right)\right)$ is defined by
$$
\forall \,  t \geq 0, \quad S_2 \ast S_1(t) = \int_0^t {S_2(s) \, S_1(t-s) \, ds} .
$$
When $S_1 = S_2$ and $X_1 = X_2 = X_3$, $S^{(\ast \ell)}$ is defined recursively by $S^{(\ast 1)} = S$ and \linebreak$S^{(\ast \ell)} = S ^{(\ast (\ell-1))}$ for any $\ell \geq 2$.

\medskip

\subsection{The abstract theorem}

Let us now present an enlargement of the functional space of a quantitative spectral mapping theorem (in the sense of semigroup decay estimate). The aim is to enlarge the space where the decay estimate on the semigroup holds. The version stated here comes from \cite[Theorem~2.13]{GMM} and \cite[Lemma~2.17]{GMM}.

\begin{theo}
\label{th:spectralgap}
Let $E$, $\EE$ be two Banach spaces such that $E \subset \EE$ with dense and continuous embedding, and consider $L \in \CCC(E)$, $\LL \in \CCC(\EE)$ with  $\mathcal{L}_{|E} = L$ and $a \in \R$. 
We assume:
\begin{enumerate}
\item[{\bf (1)}]  $L$ generates a semigroup $e^{tL}$ and
$$ 
\Sigma(L) \cap \Delta_a = \left\{\xi_1, \dots ,\xi_k\right\} \subset \Sigma_d(L) 
$$
(with $\xi_k \neq \xi_{k'}$ if $k \neq k'$ and $\left\{\xi_1, \dots ,\xi_k\right\} = \emptyset$ if $k=0$) and $L-a$ is dissipative on $\mathrm{R} \left(\mathrm{Id} - \Pi_{L,a}\right)$.
\item[{\bf (2)}]  There exist $\AA, \, \BB \in \CCC(\EE)$ such that $\LL = \AA + \BB$ (with corresponding restrictions $A$ and $B$ on $E$) and some constants $\ell_0 \in \mathbb{N}^{\ast}$, $C \geq  1$, $b \in \mathbb{R}$ and $\gamma \in [0,1)$ so that  
\begin{enumerate}
\item[{\bf (i)}] $B-a$ and $\BB - a$ are dissipative respectively on $E$ and $\EE$,
\item[{\bf (ii)}] $A \in \BBB(E)$ and $\AA \in \BBB(\EE)$, 
\item[{\bf (iii)}] $T_{\ell_0} := \left(\AA e^{\BB t}\right)^{(\ast \ell_0)}$ satisfies 
$$
\forall \,t \geq 0, \quad \left\|T_{\ell_0}(t)\right\|_{\BBB(\EE,E)} \leq C \frac{e^{bt}}{t^\gamma}.
$$
\end{enumerate}
\end{enumerate}
Then the following estimate on the semigroup holds:
$$
\forall \, a'>a, \, \forall \,  t \geq0, \quad \left\|e^{\LL t} - \sum_{j=1}^k e^{Lt} \Pi_{\mathcal{L}, \xi_j} \right\|_{\BBB(\EE)} \leq C_{a'} e^{a't} \, .
$$
\end{theo}

\begin{rem}
The assumption (2)-(iii) implies that for any $a'>a$, there exist some constructive constants $n \in \mathbb{N}$, $C_{a'} \geq 1$ such that 
$$ 
\forall \, t \geq 0, \quad \|T_{n}(t)\|_{\BBB(\EE,E)} \leq C_{a'} e^{a't} . 
$$
\end{rem}



\bigskip


\section{Semigroup decay in $L^2(\mu^{-1/2})$ where $\mu$ is the steady state} 
\label{sec:L2}
\setcounter{equation}{0}
\setcounter{theo}{0}


\subsection{Preliminaries on steady states}
We recall results obtained in~\cite{GI} about existence of steady states. They prove such a theorem for a more general equation than ours:
\begin{align*} 
\partial_t f = \, \mathcal{I} & \left[f\right] + \text{div}\left(f \nabla V\right) \quad x\in\mathbb{R}^d, \, t>0  \\
&f(0,x)=f_0(x) \qquad x \in \mathbb{R}^d  
\end{align*} 
where $f_0 \in L^1\left(\mathbb{R}^d\right)$. The operator $\mathcal{I}$ is a L\'{e}vy operator defined as:
$$
\mathcal{I}\left[f\right](x)= \text{div}\left(\sigma \nabla f \right)(x) - b \cdot \nabla f(x) + \int_{\mathbb{R}^d} {\left(f(x+z)-f(x)-\nabla f(x) \cdot z h(z)\right) \, \nu(dz) } 
$$ 
where $\sigma$ is a symmetric semi-definite $d \times d$ matrix, $b \in \R^d$ and $\nu$ denotes a nonnegative singular measure on $\R^d$ that satisfies $\nu\left(\left\{0\right\}\right)=0$ and $\int_{\mathbb{R}^d} {\min \left(1,\left|z\right|^2\right) \nu(dz)} < \infty$.  

The fractional Laplacian corresponds to a particular L\'{e}vy operator. Indeed, with  $\sigma=0$, $b=0$ and $\nu(dz) = \left|z\right|^{-d-\alpha} \, dz$, we obtain the fractional Laplacian. In this particular case, the proof of existence of steady states of~(\ref{eq:FP}) is easier, we hence give a sketch of a proof of it (it is adapted from the proof of~\cite[Theorem 1]{GI}). 

We suppose that $\mu$ is an equilibrium of the equation~(\ref{eq:FP}). At least formally, we have:
\begin{equation}
\label{eq:equilibrium}
I\left[\mu\right] + \text{div}(x\mu) = 0.
\end{equation} 
We do the following computation in order to take the Fourier transform of~(\ref{eq:equilibrium})
\begin{align*} 
\mathcal{F} \left(\text{div}(x\mu)\right) \left( \xi \right) 
&= \sum_{j=1}^d \mathcal{F} \left( \partial_j \left( x_j \mu \right) \right) \left(\xi\right) = \sum_{j=1}^d \textit{i } \xi_j \mathcal{F} \left( x_j \mu \right) \left( \xi \right) \\
&= - \sum_{j=1}^d \xi_j \partial_j \widehat{\mu}  \left( \xi \right) = - \xi \cdot \nabla \widehat{\mu} \left( \xi \right).
\end{align*} 
We deduce that an equilibrium $\mu$ satisfies 
$$
\left|\xi\right| ^\alpha \widehat{\mu} \left( \xi \right) + \xi \cdot \nabla \widehat{\mu} \left( \xi \right) = 0,
$$
which implies that $\widehat{\mu}(\xi) = C e^{ -\left|\xi\right|^\alpha/\alpha }$ for a constant $C$.

In the remaining part of the paper, we denote $\mu$ the only steady state of~(\ref{eq:FP}) of mass~$1$: $\mu \approx \FF^{-1} \left( e^{-|\cdot|^\alpha/\alpha} \right)$. 

\begin{rem} \label{rem:mu}
In the case of the classical Fokker-Planck equation~(\ref{eq:FP0}), the steady state is a Maxwellian, it is hence a Schwartz function. In our case, the steady state is not anymore a Schwartz function because its Fourier transform has a singularity at $0$. If we denote $\chi_1$ a smooth function which is nonnegative, supported on $|x| \leq 2$ and such that $\chi_1(x) = 1$ for $|x| \leq 1$, we can write the following decomposition of $\widehat{\mu}$: 
$$
\widehat{\mu}(\xi) = \chi_1(\xi) (1 + a_1 \, |\xi|^\alpha + a_2 \, |\xi|^{2\alpha} + \dots) + (1-\chi_1(\xi))  e^{ -\left|\xi\right|^\alpha/\alpha}.
$$
We see that the second part of the right-hand side is a Scwhartz function and the first one induces a singularity at $0$. We can hence prove that
$$
\mu(x) \approx |x|^{-d-\alpha} \quad \text{when} \quad |x| \rightarrow \infty.
$$
\end{rem}

\medskip

\subsection{Decay properties in $L^2(\mu^{-1/2})$}
We again use results obtained in~\cite{GI}. We just use them in our particular case, the fractional Laplacian. 

For $\Phi$ a convex function, we introduce $D_\Phi$ on $\left(\mathbb{R}^{+}\right)^2$ as:
$$
D_\Phi (a,b) = \Phi(a) - \Phi(b) - \Phi'(b) (a-b),
$$
which is nonnegative on $\left(\mathbb{R}^{+}\right)^2$.

We will not prove the next two lemmas which are going to enable us to prove the decay towards equilibrium in $L^2(\mu^{-1/2})$. The first one is~\cite[Proposition~1]{GI} and the second one is \cite[~Theorem~2]{GI} and comes from~\cite{Chaf}. 

\begin{lem}
Consider $f_0$ a nonnegative initial data for the equation~(\ref{eq:FP}) which satisfies $Ent_\mu^{\Phi}\left( \frac{f_0}{\mu} \right) < \infty$. Then, for any smooth convex function $\Phi$ and for any $t \geq 0$, the solution $f(t)$ satisfies
$$
\frac{d}{dt} E_\Phi \left(f_0 \right)(t) = - \int_{\R^d} \int_{\R^d} D_{\Phi} \left(u(t,x),u(t,x-z)\right) \frac{dz}{\left|z\right|^{d+\alpha}}  \, \mu(x) \, dx
$$
where $u(t,x)=f(t,x)/{\mu(x)}$.
\end{lem}

\begin{lem}
Let us suppose that $\Phi$ is a smooth convex function such that
$$
(a, b) \mapsto D_\Phi (a+b,b) \text{ is convex on } \left\{a+b \geq 0, b \geq 0\right\}. 
$$
Then, for any smooth function $v$, we have:
$$
Ent^\Phi_\mu \left(v(t,\cdot)\right) \leq K \int_{\mathbb{R}^d} {\int_{\mathbb{R}^d} {D_\Phi(v(t,x), v(t,x+z)) \frac{dz}{\left|z\right|^{d+\alpha}} \,  \mu(x) \, dx}} 
$$
for some $K>0$.
\end{lem}

We can now state the main theorem (\cite[Theorem~1]{GI}) of this section, its proof is a direct consequence of the two previous lemmas and the Gronwall lemma. 
\begin{theo}
\label{th:entropie}
Consider a nonnegative initial data $f_0$ such that $Ent^\Phi_\mu\left(\frac{f_0}{\mu}\right) < \infty$. We then have: 
$$ 
\forall \, t \geq 0, \quad Ent^\Phi_\mu\left(\frac{f(t)}{\mu}\right) \leq e^{-t/K} \, Ent^\Phi_\mu\left(\frac{f_0}{\mu}\right) . 
$$
\end{theo}

In what follows, we denote $\widetilde{\mu}(x) := \langle x \rangle^{-d-\alpha}$. We now give a corollary of this theorem which gives the decay property in the space $L^2(\widetilde{\mu}^{-1/2})$ i.e $L^2(\langle x \rangle^{(d+\alpha)/2})$. 
\begin{cor}
\label{cor:entropie}
Consider a nonnegative initial data $f_0$ such that $\left\|f_0 - \mu \left\langle f_0\right\rangle\right\|_{L^2\left(\widetilde{\mu}^{-1/2}\right)}$ is finite. Then, there exist $\lambda>0$ and $C>0$ such that:
$$
\left\|f(t) - \mu \left\langle f_0\right\rangle\right\|_{L^2\left(\widetilde{\mu}^{-1/2}\right)} 
\leq C \, e^{-\lambda t} \left\|f_0 - \mu \left\langle f_0\right\rangle\right\|_{L^2\left(\widetilde{\mu}^{-1/2}\right)} . 
$$
\end{cor}

\begin{proof}
Theorem~\ref{th:entropie} applied with $\Phi(s) = \left(s-\left\langle f_0\right\rangle\right)^2$ gives the result in $L^2(\mu^{-1/2})$ and the conclusion of the Corollary is a direct consequence of it because of Remark~\ref{rem:mu}.
\end{proof}

\bigskip


\section{Semigroup decay in $L^1(\langle x \rangle^k)$} 
\label{sec:L1}
\setcounter{equation}{0}
\setcounter{theo}{0}


\subsection{Splitting of the operator}
We would like to get a splitting of our operator $\LL$ into two operators which satisfies hypothesis of Theorem~\ref{th:spectralgap} with $E=L^2(\widetilde{\mu}^{-1/2})$ and $\EE=L^1(\langle x \rangle^k)$ with $k<\alpha$. In what follows, we denote $m(x) := \langle x \rangle^k$, $k<\alpha$.
\begin{lem} \label{lem:dissipL1}
Consider $a \in (-\min(k,\lambda),0)$ where $\lambda>0$ is defined in Corollary~\ref{cor:entropie}. There exist two operators $\AA$ and $\BB$ which satisfy the following conditions:
\begin{itemize}
\item[{\bf (i)}]  $\LL = \AA + \BB$,
\item[{\bf (ii)}]  $\AA \in \BBB(L^2(\widetilde{\mu}^{-1/2}))$ and  $\AA \in \BBB(L^1(m))$,
\item[{\bf (iii)}] $\BB - a$ is dissipative on $L^2(\widetilde{\mu}^{-1/2})$ and $L^1(m)$.  
\end{itemize}
\end{lem} 

\begin{proof}
We are going to estimate the integral $\int \LL f \, \text{sign}(f) \, m$ with $f$ a Schwartz function. The inequality obtained will also hold for any $f \in L^1(m)$ because of the density of $\SS(\R^d)$ in $L^1(m)$. We split the integral into two parts:

\begin{align*}
\int_{\R^d} \LL f  \, \text{sign}(f) \, m 
&= \int_{\R^d} I [f] \, \text{sign}(f) \, m + \int_{\R^d} \text{div}(xf) \, \text{sign}(f) \, m \\
&:= T_1 + T_2.
\end{align*}

As far as $T_1$ is concerned, we introduce the function $\Phi(s):= |s|$ on $\R^d$ which is convex and its derivative is $\Phi'(s)=\text{sign}(s)$. We also introduce the notation $K(x) := |x|^{-d-\alpha}$. Let us do the following computation:
\begin{align*}
&\quad \int_{\R^d} (f(y)-f(x)) \, K(x-y) \, dy \, \,  \text{sign}(f(x)) \\
&=  \int_{\R^d} \left( (f(y)-f(x)) \Phi'(f(x)) + \Phi(f(x)) - \phi(f(y)) \right) \, K(x-y) \, dy \\
&\quad + \int_{\R^d} \left( \Phi(f(y)) - \Phi(f(x)) \right) \, K(x-y) \, dy \\
&\leq \int_{\R^d} \left( |f|(y) - |f|(x) \right) \, K(x-y) \, dy = I[|f|](x),
\end{align*}
where the last inequality comes from the convexity of $\Phi$. 
We hence deduce that 
$$
T_1 \leq \int_{\R^d} I[|f|] \, m = \int_{\R^d} |f| \, I[m] = \int_{\R^d} |f| \, m \, \frac{I[m]}{m},
$$
because of Lemma~\ref{lem:IPP}. 

Let us now deal with $T_2$. Performing integrations by parts, we obtain:
\begin{align*}
T_2 &=\int_{\R^d} \text{div}(xf) \, \text{sign}(f) \, m \\
&= d \, \int_{\R^d} |f| \, m + \int_{\R^d} x \cdot \nabla f \, \text{sign} f \, m \\
&= d \, \int_{\R^d} |f| \, m + \int_{\R^d} x \cdot \nabla |f| \, m \\
&= d \, \int_{\R^d} |f| \, m -d \, \int_{\R^d} |f| \, m - \int_{\R^d} |f| \, x \cdot \nabla m \\
&= - \int_{\R^d} |f| \, m \, \frac{x \cdot \nabla m}{m}
\end{align*}

We now introduce $\psi_{m,1} := I\left[m\right]/m - x \cdot \nabla m/m$. Let us study the behavior of $\psi_{m,1}$ at infinity. First, $x \cdot \nabla m(x)/m(x)$ tends to $k$ as $|x|$ tends to infinity. Then, we prove that $I[m](x)/m(x)$ tends to $0$ as $|x|$ tends to infinity. We use both representations~(\ref{eq:LF}) and~(\ref{eq:LF2}) to split $I[m](x)$ into two parts:
\begin{align*}
I[m](x) &= \frac{1}{2} \int_{|z| \leq 1} \left(m(x+z) + m(x-z) - 2m(x)\right) \, K(z) \, dz \\
&\quad +  \int_{|x-y| \geq 1} \left(m(y) - m(x)\right) \, K(x-y) \, dy \\
&:=  I_1[m](x) + I_2[m](x).
\end{align*}

Concerning $I_1[m]$, using a Taylor expansion, we obtain:
\begin{align*}
|m(x+z) + m(x-z) - 2m(x)| &\leq \sup_{|z| \leq 1} \|D^2m(x+z)\|_{\infty} \, |z|^2 \\
&\leq C \langle x \rangle^{k-2} |z|^2,
\end{align*}
from which we deduce that 
\beqn \label{eq:I1(m)}
I_1[m](x) \leq C \langle x \rangle^{k-2} \int_{|z| \leq 1} \frac{dz}{|z|^{d+\alpha-2}}.
\eeqn

Concerning $I_2[m]$, let us introduce the function $\psi(s) := s^{k/2}$ on $\R^+$. Using the fact that $\psi$ is $k/2$-H\"{o}lder continuous on $\R^+$ because $k/2 \leq 1$, we obtain for any $x$, $y \in \R^d$:
$$
\left|\psi(1+ |x|^2) - \psi(1+|y|^2)\right| \leq C \left| |x|^2-|y|^2 \right|^{k/2}
$$
for some $C>0$. We deduce the following inequalities:
\begin{align*}
|m(x)-m(y)| &\leq C \, ||x|-|y||^{k/2} (|x|+|y|)^{k/2} \\
&\leq C \, |x-y|^{k/2} (|x|+|x-y|+ |x|)^{k/2} \\
&\leq 2\, C \left(|x-y|^{k/2} |x|^{k/2} +  |x-y|^k \right).
\end{align*}
Finally, we obtain the following estimate on $I_2[m]$:
\beqn \label{eq:I2(m)}
I_2[m](x) \leq C \left(|x|^{k/2} \int_{|z|\geq1} \frac{dz}{|z|^{d+\alpha-k/2}} +  \int_{|z|\geq1} \frac{dz}{|z|^{d+\alpha-k}}\right),
\eeqn
where we notice that the integrals are convergent because $k<\alpha$.

Gathering~(\ref{eq:I1(m)}) and~(\ref{eq:I2(m)}), we deduce that $I[m]/m$ tends to $0$ at infinity. Finally, we obtain:
$$
 \int_{\R^d} \LL f \, \text{sign}f\,  m \leq \int_{\R^d} |f| \, m \, \psi_{m,1} \quad \text{with} \quad \lim_{|x| \rightarrow \infty} \psi_{m,1}(x) = -k<0.
$$ 

We introduce the smooth function $\chi_R$ ($R>0$) which is nonnegative, supported on $|x| \leq 2R$ and such that $\chi_R(x) = 1$ for $|x| \leq R$.  For any $a >-\min(\lambda,k)$, we may find $M$ and $R$ large enough so that
\beqn \label{eq:psi}
\forall \, x \in \R^d, \quad \psi_{m,1}(x) - M \chi_R (x) \leq a.
\eeqn
Indeed, if we choose $R$ large enough such that for any $|x| \geq R$, $\psi_{m,1}(x) \leq a$ and \linebreak $M := \max_{|x|\leq R} \psi_{m,1}(x) - a$, we have~(\ref{eq:psi}). 

We then introduce $\AA := M\chi_R$ and $\BB := \LL - M \chi_R$. We finally obtain:
$$
\begin{aligned}
\int_{\R^d} \left(\BB-a\right) f \, \text{sign} f \, m &= \int_{\R^d} \left(\LL - M \chi_R - a\right) f \, \text{sign} f \, m \\
&\leq \int_{\R^d} \left(\psi_{m,1} - M \chi_R - a\right) \, |f| \, m \\
&\leq 0,
\end{aligned}
$$
which implies that $\BB-a$ is dissipative on $L^1(m)$. 

Let us now check that $\BB-a$ is dissipative on $L^2(\widetilde{\mu}^{-1/2})$: 
$$
\begin{aligned}
\int_{\R^d} \BB f \, f \,  \mu^{-1} 
&= \int_{\R^d} \LL f \, f \,  \mu^{-1} - M \int_{\R^d} \chi_R \, f^2  \mu^{-1} \\
&\leq \int_{\R^d} \LL f \, f \,  \mu^{-1} \\
&= \int_{\R^d} \LL (f - \mu \langle f \rangle) \, f\,  \mu^{-1} + \int_{\R^d} \langle f \rangle \, \LL \mu \, \mu^{-1} \\
&=  \int_{\R^d} \LL (f - \mu \langle f \rangle) \, (f - \mu \langle f \rangle)\,  \mu^{-1} \\
&\leq -\lambda \, \|f\|^2_{L^2(\mu^{-1/2})},
\end{aligned}
$$
where the last inequality comes from Corollary~\ref{cor:entropie}. We thus deduce that $\BB-a$ is dissipative on $L^2(\mu^{-1/2})$ using that $-\lambda \leq a$. To conclude that we also have the dissipativity of $\BB-a$ on $L^2(\widetilde{\mu}^{-1/2})$, we use the fact that there exists $C>0$ such that $\mu \leq C \widetilde{\mu}$ on $\R^d$.

We can now conclude. This splitting $\LL = \AA + \BB$ fulfills conditions (i), (ii) and (iii) of Lemma~\ref{lem:dissipL1}. Indeed, it is immediate to check assumption (ii) because $\AA$ is a truncation operator. 
\end{proof}

\medskip 

\subsection{Regularization properties of $\left( \AA \, e^{\BB t} \right)^{(*n)}$}
We are now going to show that there exists $n \in \N$ such that $\left(\AA \, e^{\BB t}\right)^{(*n)}$ has a regularizing effect. In order to get such a result, we are going to use the negative term  in the computations done to get the dissipativity of $\BB$. Let us state a result which is going to be useful to get an estimate on this negative term. 

\begin{lem}[Fractional Nash inequality] \label{lem:nash}
Consider $\alpha \in (0,2)$. There exists a constant $C>0$ such that for any $g \in L^1(\R^d) \cap H^{\alpha/2}(\R^d)$, we have:
$$
\int_{\mathbb{R}^d} {\left|g(x)\right| ^2 \, dx} \leq C \, \left(\int_{\mathbb{R}^d} {\left|g(x)\right|  \, dx}\right)^{\frac{2\alpha}{d+\alpha}} 
\left(\int_{\mathbb{R}^d} {\int_{\mathbb{R}^d} {\frac{\left|g(x) - g(y)\right|^2}{\left|x-y\right|^{d+\alpha}}\, dy}\, dx} \right)^{\frac{d}{d+\alpha}} .
$$
\end{lem}

\begin{proof}
We use the Plancherel formula to get the following equality for any $R>0$:
$$
\int_{\mathbb{R}^d} {\left|g(x)\right| ^2 \, dx} = C \left(\int_{|\xi| \leq R} {\left|\widehat{g}(\xi)\right| ^2 \, d\xi} + \int_{|\xi| \geq R} {\left|\widehat{g}(\xi)\right| ^2 \, d\xi}\right) .
$$
The first part of the integral can be bounded as follows:
$$ 
\int_{|\xi| \leq R} {\left|\widehat{g}(\xi)\right| ^2 \, d\xi} \leq R^d \left\|\widehat{g}\right\|_{L^\infty}^2 \leq  R^d \left\|g\right\|_{L^1}^2. 
$$
As far as the second part is concerned, we use Lemma~\ref{lem:sobolevfrac}
$$
\int_{|\xi| \geq R} {\left|\widehat{g}(\xi)\right| ^2 \, d\xi} \leq \frac{1}{R^\alpha} \int_{\mathbb{R}^d} {\left|\xi\right|^{\alpha} \left|\widehat{g}\left(\xi\right)\right|^2 \, d\xi} = \frac{C}{R^\alpha} \int_{\mathbb{R}^d} {\int_{\mathbb{R}^d} {\frac{\left(g(x) - g(y)\right)^2}{\left|x-y\right|^{d+\alpha}}}} \, dy \, dx. 
$$
We denote
$$
a = \left\|g\right\|_{L^1}^2 \text{ et } b = C \int_{\mathbb{R}^d} {\int_{\mathbb{R}^d} {\frac{\left(g(x) - g(y)\right)^2}{\left|x-y\right|^{d+\alpha}}}} \, .
$$
and the aim is to minimize the function $\phi(R) := a R^d + bR^{-\alpha}$ to get an optimal inequality. We compute
$$ 
\phi'(R) = 0 \Longleftrightarrow adR^{d-1} - \alpha b \frac{1}{R^{\alpha +1 }} = 0
\Longleftrightarrow R = \left(\frac{\alpha b}{a d }\right)^{\frac{1}{d+\alpha}} 
$$
and
$$
\phi \left(\left(\frac{\alpha b}{a d }\right)^{\frac{1}{d+\alpha}}\right) = C a^{\frac{\alpha}{d+\alpha}} b^{\frac{d}{d+\alpha}},
$$
which concludes the proof.
\end{proof}

Let us now prove the following lemma which is the cornerstone of the proof of the regularizing effect of $\left(\AA \, e^{\BB t}\right)^{(*n)}$. We introduce the following measure: 
$$
m_0(x) := \langle x \rangle^{k_0} \quad \text{with} \quad k_0 < \min(k, \alpha/2).
$$
Let us notice that this assumption on $k_0$ allows us to define $I[m_0^p]$ for any $p \in [1,2]$ and that $m_0$ satisfies $L^2(\widetilde{\mu}^{-1/2}) \subset L^q(m_0)$ for any $q \in [1,2]$.  
\begin{lem}
\label{lem:regular}
There are $b,C>0$ such that for any $p$ and $q$, $1\leq p \leq q \leq 2$, we have:
$$
\forall \, t \geq 0, \quad \left\| e^{\BB t} \, f \right\|_{L^q(m_0)} \leq \frac{C e^{bt}}{t^{\frac{d}{\alpha} \left(\frac{1}{p} - \frac{1}{q} \right)}} \left\|f\right\|_{L^p(m_0)} .
$$
\end{lem}

\begin{proof}
For $p\in [1,2]$, we denote 
$$
\psi_{m_0,p} := \frac{I[m_0^p]}{p\, m_0^p} + d\, \frac{p-1}{p} - \frac{x \cdot \nabla(m_0^p)}{p\, m_0^p}
$$
and we introduce $b \in \R$ such that $\sup_{q \in [1,2]} \psi_{m_0,q} \leq b$. 

Let us prove that for any $p \in [1,2]$, we have:
\beqn \label{eq:LpLp}
\forall \, t \geq 0, \quad \|e^{\BB t} f \|_{L^p(m_0)} \leq e^{bt} \|f\|_{L^p(m_0)}.
\eeqn 
We now do same kind of computations as in the proof of Lemma~\ref{lem:dissipL1}:
\begin{align*}
\int_{\R^d} \LL f \, |f|^{p-1} \, \text{sign}f \, m_0^p &= \int_{\R^d} I[f] \, |f|^{p-1} \, \text{sign}f \, m_0^p + \int_{\R^d} \text{div}(xf) \, |f|^{p-1} \, \text{sign}f \, m_0^p \\
&:= \widetilde{T}_1 + \widetilde{T}_2.
\end{align*}
As far as $\widetilde{T}_1$ is concerned, we introduce the function $\Phi(x):= |x|^p/p$ on $\R^d$ which is convex and its derivative is $\Phi'(x)=|x|^{p-1} \text{sign}(x)$. Let us do the following computation:
\begin{align*}
& \quad \int_{\R^d} (f(y)-f(x)) \, K(x-y) \, dy \,  |f|^{p-1}(x) \, \text{sign}(f(x)) \\
&=  \int_{\R^d} \left( (f(y)-f(x)) \Phi'(f(x)) + \Phi(f(x)) - \phi(f(y)) \right) \, K(x-y) \, dy \\
&\quad+ \int_{\R^d} \left( \Phi(f(y)) - \Phi(f(x)) \right) \, K(x-y) \, dy \\
&\leq \int_{\R^d} \frac{1}{p} \left( |f|^p(y) - |f|^p(x) \right) \, K(x-y) \, dy = \frac{1}{p}I[|f|^p](x),
\end{align*}
where the last inequality comes from the convexity of $\Phi$. 
We hence deduce that 
$$
\widetilde{T}_1 \leq \int_{\R^d} I[|f|^p] \, m_0 = \frac{1}{p}\int_{\R^d} |f|^p \, I[m_0^p] = \int_{\R^d} |f|^p \, m_0^p \, \frac{I[m_0^p]}{p \, m_0^p}.
$$
Concerning $\widetilde{T}_2$, using an integration by part, we obtain:
$$
\widetilde{T}_2 = \int_{\R^d} \left[d \, \frac{p-1}{p} - \frac{x \cdot \nabla (m_0^p)}{p\, m_0^p} \right] \, |f|^p \, m_0^p.
$$
Finally, the previous estimates imply that 
$$
\int_{\R^d} \BB f \, |f|^{p-1} \, \text{sign}f \, m_0^p \leq \int_{\R^d} (\psi_{m_0,p} - M \chi_R)\, |f|^p \, m_0^p \leq b \int_{\R^d} |f|^p \, m_0^p
$$
using the definition of $b$. This implies the estimate~(\ref{eq:LpLp}).

In order to establish the gain of integrability estimate, we have to use the non positive term in a sharper way, i.e. not merely
the fact that it is non-positive. It is enough to do that in the simplest case when $p = 2$.

Let us consider a solution $f_t$ of the equation 
$$ 
\partial_t f_t = \mathcal{B} f_t \, , \quad f_0 = f \in L^2(m_0) .
$$
The previous computation involving the function $\Phi(x)$ is simpler in the case $p=2$ and becomes:
\begin{align*}
\int_{\R^d} \BB f \, f \, m_0^2 
&= -\frac{1}{2} \int_{\R^d} {\int_{\R^d} {\frac{\left|f(x) - f(y)\right|^2}{\left|x-y\right|^{d+\alpha}} \, dy} \, m_0^2(x) \, dx} + \int f^2 \, m_0^2 \left( \psi_{m_0,2} - M \chi_R \right) \\
& \leq -\frac{1}{2} \int_{\R^d} {\int_{|x-y| \leq 1} {\frac{\left|f(x) - f(y)\right|^2}{\left|x-y\right|^{d+\alpha}} \, dy} \, m_0^2(x) \, dx} + b \int f^2 \, m_0^2
\end{align*}
Let us deal with the negative part of the last inequality. 
\begin{align*}
& \quad  \int_{\mathbb{R}^d} {\int_{\left|x-y\right|\leq 1} {\frac{\left|f_t(x) - f_t(y)\right|^2}{\left|x-y\right|^{d+\alpha}} \, dy} \, m_0^2(x) \, dx} \\
&=   \int_{\mathbb{R}^d} {\int_{\left|x-y\right|\leq 1} {\frac{\left|f_t(x)m_0(x) - f_t(y)m_0(y) + f_t(y) \left(m_0(y) - m_0(x)\right)\right|^2}{\left|x-y\right|^{d+\alpha}} \, dy} \, m_0^2(x) \, dx} \\
&\geq  \frac{1}{2} \int_{\mathbb{R}^d} {\int_{\left|x-y\right|\leq 1} {\frac{\left|f_t(x)m_0(x) - f_t(y)m_0(y)\right|^2}{\left|x-y\right|^{d+\alpha}} \, dy} \, m_0^2(x) \, dx} \\
&\quad - \int_{\mathbb{R}^d} {\int_{\left|x-y\right|\leq 1} {\frac{\left|m_0(x) - m_0(y)\right|^2}{\left|x-y\right|^{d+\alpha}} \, dx} \, f_t^2(y) \, dy} \\
&\geq \frac{1}{2} \int_{\mathbb{R}^d} {\int_{\mathbb{R}^d} {\frac{\left|f_t(x)m_0(x) - f_t(y)m_0(y)\right|^2}{\left|x-y\right|^{d+\alpha}} \, dy} \, m_0^2(x) \, dx} \\
&\quad - \frac{1}{2} \int_{\mathbb{R}^d} {\int_{\left|x-y\right|\geq 1} {\frac{\left|f_t(x)m_0(x) - f_t(y)m_0(y)\right|^2}{\left|x-y\right|^{d+\alpha}} \, dy} \, m_0^2(x) \, dx}  \\
&\quad- \int_{\mathbb{R}^d} {\int_{\left|x-y\right|\leq 1} {\frac{\left|m_0(x) - m_0(y)\right|^2}{\left|x-y\right|^{d+\alpha}} \, dx} \, f_t^2(y) \, dy} 
\end{align*}
We treat the first term using Lemma~\ref{lem:nash} with $g = f_t m_0$:
\begin{equation} 
\label{eq:nash1}
\int_{\mathbb{R}^d} {\int_{\mathbb{R}^d} {\frac{\left|f_t(x)m_0(x) - f_t(y)m_0(y)\right|^2}{\left|x-y\right|^{d+\alpha}} \, dy} \,  dx} 
\geq C \left(\int_{\mathbb{R}^d} {\left|f_t\right|^2 m_0^2}\right)^{\frac{d+\alpha}{d}} \, \left(\int_{\mathbb{R}^d} {\left|f_t\right| m_0} \right) ^{-\frac{2 \alpha}{d}}  .
\end{equation}
We crudely bound the second term from above:
\beqn  \label{eq:nash2}
\begin{aligned}
& \quad \int_{\mathbb{R}^d} {\int_{\left|x-y\right|\geq 1} {\frac{\left|f_t(x)m_0(x) - f_t(y)m_0(y)\right|^2}{\left|x-y\right|^{d+\alpha}} \, dy} \, m_0^2(x) \, dx} \\
&\leq   C \left(\int_{\mathbb{R}^d} {\int_{\left|x-y\right|\geq 1} {\frac{\left|f_t(x)m_0(x)\right|^2}{\left|x-y\right|^{d+\alpha}} \, dy} \, dx} +
\int_{\mathbb{R}^d} {\int_{\left|x-y\right|\geq 1} {\frac{\left|f_t(y)m_0(y)\right|^2}{\left|x-y\right|^{d+\alpha}} \, dx} \, dy} \right) \\
&\leq  C \int_{\mathbb{R}^d} {\left|f_t\right|^2 m_0^2}. 
\end{aligned}
\eeqn
Finally, the third term is bounded using the fact that $\sup_{\bar{B}(y,1)} | \nabla m_0|^2 \leq  C \, m_0^2(y)$:
\beqn  \label{eq:nash3}
\begin{aligned}
& \quad \int_{\mathbb{R}^d} {\int_{\left|x-y\right|\leq 1} {\frac{\left|m_0(x) - m_0(y)\right|^2}{\left|x-y\right|^{d+\alpha}} \, dx} \, f_t^2(y) \, dy} \\ 
&\leq  C \int_{\mathbb{R}^d} {\int_{\left|x-y\right|\leq 1} \frac{|x-y|^2 \sup_{\bar{B}(y,1)}  |\nabla m_0|^2 }{\left|x-y\right|^{d+\alpha}} \, dx} \, f_t^2(y) \, dy \\
&\leq  C \int_{\mathbb{R}^d} {\int_{\left|z\right|\leq 1} {\frac{1}{\left|z\right|^{d+\alpha-2}} \, dz} \, f_t^2(y) \, m_0^2(y) \, dy} \\ 
&\leq C \int_{\mathbb{R}^d} f_t^2 \, m_0^2.
\end{aligned}
\eeqn
Gathering~(\ref{eq:nash1}),~(\ref{eq:nash2}) et~(\ref{eq:nash3}), we obtain:
\beqn  \label{ineq0}
\begin{aligned}
&\int_{\mathbb{R}^d} {\int_{\left|x-y\right|\leq 1} {\frac{\left|f_t(x) - f_t(y)\right|^2}{\left|x-y\right|^{d+\alpha}} \, dy} \, m_0^2(x) \, dx} \\
&\geq C \left(\int_{\mathbb{R}^d} {\left|f_t\right|^2 m_0^2}\right)^{\frac{d+\alpha}{d}} \, \left(\int_{\mathbb{R}^d} {\left|f_t\right| m_0} \right) ^{-\frac{2 \alpha}{d}} 
- C' \left(\int_{\mathbb{R}^d} {f_t^2 m_0^2}\right),
\end{aligned} 
\eeqn
for some constants $C$, $C'>0$.
We introduce the following notations:
$$
X(t) := \left\|f_t\right\|_{L^2(m_0)}^2 \text{ et } Y(t) := \left\|f_t\right\|_{L^1(m_0)}.
$$

On the one hand, if $X_0 \leq \left(2C'/C\right)^{d/\alpha} Y_0^2$, because of estimate~(\ref{eq:LpLp}), we have: \linebreak $\forall \, t \geq 0, \, X(t)^{1/2} \leq C e^{bt} X_0^{1/2}$. We hence obtain
$$
\forall \, t \geq 0, \quad X(t)^{1/2} \leq C e^{bt} Y_0.
$$

On the other hand, we treat the case $X_0> \left(2C'/C\right)^{d/\alpha} Y_0^2$. By the previous step (\ref{ineq0}), we end up with the differential inequality
\begin{equation}
\label{ineq1}
\frac{d}{dt} X(t) \leq - C \, Y(t)^{-\frac{2 \alpha}{d}} X(t)^{1+\frac{\alpha}{d}} + C' \, X(t)\,. 
\end{equation}
We also have from estimate~(\ref{eq:LpLp}): $ Y(t) \leq Ce^{bt} Y(0)$ for any $t \geq0$. So, we obtain for any  $t \in [0,1]$, $Y(t) \leq C \, Y(0)$ changing the value of $C$. Putting this together with~(\ref{ineq1}), we obtain:
\begin{equation}
\label{ineq2} 
\forall \,t \in [0,1], \quad \frac{d}{dt} X(t) \leq - C \, Y_0^{-\frac{2 \alpha}{d}} X(t)^{1+\frac{\alpha}{d}} + C' \, X(t).
\end{equation}

Let us introduce $\tau := \sup \left\{t \in [0,1] \, : \, X(s) \geq \left(2C'/C\right)^{d/\alpha} Y_0^2, \, \, \forall \, s \in [0,t] \right\}$. For any \linebreak $t \in \, ]0,\tau[$, we have $-1/2 \, C \, X(t) ^{1+\alpha/d} Y_0^{-2\alpha /d} \leq -C' X(t)$. Then, using~(\ref{ineq2}), we obtain: 
$$ 
\forall \, t \in \, (0,\tau), \quad \frac{d}{dt} X(t) \leq - \frac{1}{2} C \, Y_0^{-\frac{2 \alpha}{d}} X(t)^{1+\frac{\alpha}{d}},
$$
which finally implies
\begin{equation} 
\label{ineg1}
\forall \, t \in (0, \tau), \quad X(t) \leq \left(\frac{\alpha}{d} \frac{C}{2} Y_0 ^{-\frac{2\alpha}{d}} t \right)^{-\frac{d}{\alpha}}.
\end{equation}
Moreover, because of estimate~(\ref{eq:LpLp}), we get:
\beqn \label{ineg2}
\forall \, t \in \, [\tau,+\infty), \quad  X(t)^{1/2} \leq C e^{b(t-\tau)} X(\tau)^{1/2}  \leq C  e^{b(t-\tau)} \left(\frac{2C'}{C}\right)^{\frac{d}{2\alpha}} Y_0. 
\eeqn
Therefore, gathering inequalities~(\ref{ineg1}) an~(\ref{ineg2}), we obtain:
$$
\forall \, t >0, \quad X(t)^{\frac{1}{2}} \leq C \, t^{-\frac{d}{2\alpha}} \, e^{b t} \, Y_0.
$$

As a conclusion, we have
$$
\forall \, t>0, \quad \left\|e^{\BB t} \, f\right\|_{L^2(m_0)} \leq C \, e^{bt}{t^{-\frac{d}{2\alpha}}} \left\|f\right\|_{L^1(m_0)}. 
$$
which means that the operator $e^{\BB t}$ is continuous from $L^1(m_0)$ into $L^2(m_0)$. 

\medskip
Let us now consider $p$ and $q$, $1 \leq p \leq q$, $e^{\BB t}$ is continuous from $L^p(m_0)$ into $L^q(m_0)$ using Riesz-Thorin interpolation Theorem. 
Moreover, if we denote $C_{ab}$ the norm of \linebreak$e^{\BB t} : L^a(m_0) \rightarrow L^b(m_0)$, we get the following estimate:
$$ 
C_{pq} \leq C_{22}^{2-2/p} \, C_{11}^{2/q - 1} \, C_{12}^{2/p-2/q} 
$$
and
\begin{align*}
C_{22}^{2-2/p} \, C_{11}^{2/q - 1} \, C_{12}^{2/p-2/q} &= C \, e^{bt(2-2/p)} \, e^{bt(2/q-1)} \, e^{bt(2/p-2/q)} \, t^{-d/(2\alpha)\, (2/p-2/q)} \\
&= \frac{Ce^{bt}}{t^{\frac{d}{\alpha}\left(\frac{1}{p} - \frac{1}{q}\right)}},
\end{align*}
which yields the result. 
\end{proof}

Using the same method as in~\cite{GMM}, we can deduce the following corollary:
\begin{cor}
\label{cor:regular}
There exists a constant $C$ such that for any $p$ and $q$, $1 \leq p \leq q \leq  2$, we have:
$$
\forall \, t \geq 0, \quad \|T_{\ell_0}(t)f\|_{L^q(m_0)} \leq C \, \frac{t^{\ell_0-1}e^{bt}}{t^{\frac{d}{\alpha}\left(\frac{1}{p} - \frac{1}{q}\right)}} \|f\|_{L^p(m_0)}
$$
where $\ell_0 = E\left[\frac{d}{\alpha}\left(\frac{1}{p} - \frac{1}{q}\right)\right]+1$.
\end{cor}

\medskip
\subsection{Proof of the main result}
As a conclusion, we can now apply Theorem~\ref{th:spectralgap} with $E=L^2(\widetilde{\mu}^{-1/2})$ and $\EE = L^1(m)$. Hypothesis (1) comes from Corollary~\ref{cor:entropie}. Hypothesis (2)-(i) and (2)-(ii) come from Lemma~\ref{lem:dissipL1}.
We can also prove that assumption (2)-(iii) is satisfied. 

Indeed, we can check by an immediate computation that we have the following estimate for any function $f$: $ \| \AA f \|_{L^q(m)} \leq C \|\AA f\|_{L^q(m_0)}$.
Moreover, we have that $L^p(m) \subset L^p(m_0)$ with continuous embedding (because $k_0 < k$). Using these two last facts and Corollary~\ref{cor:regular}, we can deduce that for any $p$ and $q$, $1 \leq p \leq q \leq  2$, we have:
\beqn \label{eq:regular}
\forall \, t \geq 0, \quad \|T_{\ell_0}(t)f\|_{L^q(m)} \leq C \, \frac{t^{\ell_0-1}e^{bt}}{t^{\frac{d}{\alpha}\left(\frac{1}{p} - \frac{1}{q}\right)}} \|f\|_{L^p(m)}.
\eeqn
Moreover, we can show that $\| \AA f \|_{L^2(\widetilde{\mu}^{-1/2})} \leq C \|\AA f\|_{L^2(m)}$.
Finally, using the last estimate combined with (\ref{eq:regular}) with $p=1$, $q=2$ and denoting $\gamma := \frac{d}{2\alpha} - E\left(\frac{d}{2\alpha}\right)$,  we obtain:
$$
\|T_{\ell_0}(t)  \|_{L^1(m) \rightarrow L^2(\widetilde{\mu}^{-1/2})} \leq C \frac{e^{bt}}{t^\gamma},
$$
with $\gamma \in [0,1)$, which implies that (2)-(iii) is fulfilled. 

We can conclude that Theorem~\ref{th:main} holds. 

\begin{rem}
To obtain a similar result as Theorem~\ref{th:main} in $L^p(\langle x \rangle^k)$ with $p \in (1,2]$, we need a very restrictive assumption: $ d\,(1-1/p)<k<\alpha$. Indeed, it implies that the limit at infinity of $\psi_{m,p}$ is negative, which allows us to get the dissipativity of $\BB-a$ in $L^p(\langle x \rangle^k)$ for any $a>d\, (1-1/p)-k$. The rest of the proof can be done in the same way.
\end{rem}

\newpage
\bibliographystyle{acm}

\end{document}